\documentclass[letter,11pt,reqno]{amsart}

\usepackage[utf8]{inputenc}

\usepackage{etex}

\usepackage{xcolor}

\definecolor{verydarkblue}{rgb}{0,0,0.5}

\usepackage[
    breaklinks,
    colorlinks,
    citecolor=verydarkblue,
    linkcolor=verydarkblue,
    urlcolor=verydarkblue,
    pagebackref=true,
    hyperindex
]{hyperref}

\backrefenglish

\usepackage{fancyhdr}

\usepackage[
    hscale=0.7,
    vscale=0.75,
    headheight=13pt,
    centering,
]{geometry}

\usepackage{amsmath}
\usepackage{amsthm}
\usepackage{amssymb}
\usepackage{mathtools}
\usepackage{bm}		%for boldface greek symbols
\usepackage{mathdots}
\usepackage{framed}
\usepackage{aliascnt}
\usepackage[capitalize]{cleveref}
\usepackage{array}
\usepackage[alphabetic,msc-links]{amsrefs}
\usepackage[all,cmtip]{xy}

\crefname{appendix}{Appendix}{Appendices}
\Crefname{appendix}{Appendix}{Appendices}
\let\oldappendix\appendix
\renewcommand{\appendix}{%
  \oldappendix%
  \crefalias{section}{appendix}%
}

\theoremstyle{plain}

\crefname{introtheorem}{Theorem}{Theorems}

\newaliascnt{introcorollary}{introtheorem}

\aliascntresetthe{introcorollary}

\crefname{introcorollary}{Corollary}{Corollaries}

\newtheorem{theorem}{Theorem}[section]
\newtheorem*{theorem*}{Theorem}
\crefname{theorem}{Theorem}{Theorems}

\newaliascnt{proposition}{theorem}
\newtheorem{proposition}[proposition]{Proposition}
\aliascntresetthe{proposition}
\crefname{proposition}{Proposition}{Propositions}

\newaliascnt{lemma}{theorem}
\newtheorem{lemma}[lemma]{Lemma}
\aliascntresetthe{lemma}
\crefname{lemma}{Lemma}{Lemmas}

\newaliascnt{corollary}{theorem}

\aliascntresetthe{corollary}
\crefname{corollary}{Corollary}{Corollaries}

\newaliascnt{problem}{theorem}

\aliascntresetthe{problem}
\crefname{problem}{Problem}{Problems}

\newaliascnt{conjecture}{theorem}

\aliascntresetthe{conjecture}
\crefname{conjecture}{Conjecture}{Conjectures}

\newaliascnt{claim}{theorem}

\aliascntresetthe{claim}
\crefname{claim}{Claim}{Claims}

\theoremstyle{definition}

\newaliascnt{definition}{theorem}

\aliascntresetthe{definition}
\crefname{definition}{Definition}{Definitions}

\newaliascnt{situation}{theorem}

\aliascntresetthe{situation}
\crefname{situation}{Situation}{Situations}

\crefname{case}{Case}{Cases}

\theoremstyle{remark}

\newaliascnt{notation}{theorem}
\newtheorem{notation}[notation]{Notation}
\aliascntresetthe{notation}
\crefname{notation}{Notation}{Notations}

\newaliascnt{remark}{theorem}
\newtheorem{remark}[remark]{Remark}
\aliascntresetthe{remark}
\crefname{remark}{Remark}{Remarks}

\newaliascnt{example}{theorem}

\aliascntresetthe{example}
\crefname{example}{Example}{Examples}

\newaliascnt{question}{theorem}

\aliascntresetthe{question}
\crefname{question}{Question}{Questions}

\newtheorem{step}{Step}
\crefname{step}{Step}{Steps}

\numberwithin{figure}{section}

\numberwithin{equation}{section}
\usepackage{marginnote}

\def\N{{\mathbb N}}
\def\Z{{\mathbb Z}}
\def\Q{{\mathbb Q}}

\def\C{{\mathbb C}}

\def\A{{\mathbb A}}

\def\cN{\mathcal{N}}

\def\fa{\mathfrak{a}}

\def\a{\alpha}
\def\b{\beta}
\def\g{\gamma}

\def\f{\varphi}

\def\e{\eta}

\def\p{\pi}

\def\G{\Gamma}
\def\S{\Sigma}

\def\.{\cdot}
\let\circum\^
\def\^{\widehat}
\def\~{\widetilde}
\def\o{\circ}

\def\({\left(}
\def\){\right)}

\def\*{{}^*}

\renewcommand{\and}{ \ \ \text{ and } \ \ }

\def\red{\mathrm{red}}

\def\sm{\mathrm{sm}}

\DeclareMathOperator{\Spec} {Spec}

\DeclareMathOperator{\Sing} {Sing}

\DeclareMathOperator{\val} {val}

\DeclareMathOperator{\ord} {ord}

\DeclareMathOperator{\chr} {char}

\begin{document}

\title[Arcs on rational double points]
{Arcs on rational double points in arbitrary characteristic}

\author{Tommaso de Fernex}
\address{Department of Mathematics, University of Utah, Salt Lake City, UT 84112, USA}
\email{{\tt defernex@math.utah.edu}}

\author{Shih-Hsin Wang}
\address{Department of Mathematics, University of Utah, Salt Lake City, UT 84112, USA}
\email{{\tt shwang@math.utah.edu}}

\dedicatory{With an appendix by Tommaso de Fernex, Roi Docampo, and 
Gebhard Martin}


\thanks{%
The research of the first author was partially supported by NSF grant DMS-2001254.
The research of the second author was partially supported by NSF RTG DMS-1840190.
}

\thanks{%
The research of the first author was partially supported by NSF Grant DMS-2001254.
}

\begin{abstract}
We settle the Nash problem for two-dimensional rational double points in all characteristics.
The proof is based on a direct computation of the families of arcs through these singularities. 
\end{abstract}

\maketitle

\section{Introduction}

In the late sixties, 
Nash uncovered a connection between certain information coming from resolution of
singularities of complex varieties and the topology of the space of formal arcs through
their singularities \cite{Nas95}.
The connection takes a particularly simple form in dimension two, where a naturally defined 
injective map, called the \emph{Nash map}, is established 
from 
the set of irreducible components of the space of arcs of the singularity
to 
the set of exceptional divisors in the minimal resolution.
This map can be interpreted as
an inclusion between certain sets of divisorial valuations, and Nash conjectured
that the map is surjective, namely, the two sets of valuations are the same. 
This conjecture became known as the \emph{Nash problem}. Nash also proposed a conjectural 
picture for higher dimensional singularities which
turned out to hold in some cases but fail in general \cite{IK03,dF13}.
The Nash problem in higher dimension eventually evolved into the question about determining 
the image of the Nash map (see, e.g., \cite{LJR12,dFD16,dFW24,BLM25} for recent works in this direction), 
a question which still remains largely open.

Pioneer work towards the Nash problem 
was done in \cite{LJ80} where an approach via \emph{wedges} 
(namely, arcs on arc spaces) to detect adjacencies in the space of arcs was first introduced, 
though a rigorous proof of the existence of such wedges, based on a result known as the \emph{curve selection lemma}, 
was only established years later in \cite{Reg06}. 
Several special cases of the 
Nash problem in dimension two were solved throughout the years in the works
\cite{Nas95,Reg95,LJR99,IK03,Reg04,PPP06,Ple08,LA11,PS12,PP13} before the problem 
was completely solved in \cite{FdBPP12}.
An independent proof for rational surface singularities was given in \cite{Reg12}. 
Surveys can be found in \cite{Ple15,dF18}. 

People have looked at the Nash problem for normal surfaces in positive characteristics as well, 
see for instance \cite{Reg95,IK03,Reg06,Reg12,Nob18}. Some cases such as 
minimal surface singularities (i.e., two-dimensional rational singularities whose fundamental cycle in the minimal resolution is reduced) and toric singularities (in all dimensions) were settled, but
beyond that, little is know about the Nash problem in positive characteristics.  
It should be noted that, except for $A_n$ singularities, 
rational double points are not covered by those results. These can be considered in many 
regards as the simplest singularities. They were classically studied, 
with their classification tracing through \cite{DV34,Art66,Lip69,GK90}, 
and appear in many different contexts.

The purpose of this paper is to settle the Nash problem for rational double points
in positive characteristics.

\begin{theorem}
\label{t:Nash-DV}
The Nash problem holds for two-dimensional rational double points over algebraically closed fields of 
arbitrary characteristic.
\end{theorem}

It is instructive to look at the history of the Nash problem for rational double points
in characteristic zero. While the case of $A_n$ singularities
can be checked with a simple computation and was already well understood by Nash, 
it took several decades before the remaining
rational double points were settled in characteristic zero. 
The cases of $D_n$ and $E_6$ singularities were treated in \cite{Ple08,PS12},
%by analyzing the conditions induced on the coefficients
%of the arcs by the equation of the singularitie. A different approach
%was introduced in \cite{LA11} to give new proofs for $D_n$ and $E_6$ singularities
and a new proof for $D_n$ and $E_6$ was given \cite{LA11}
where the $E_7$ singularity case was settled as well. 
However, the remaining case of the $E_8$ singularity
proved to be too hard to be dealt with the methods
introduced in those works. 
For this last case, a completely different route was taken in \cite{PP13}, 
where the Nash problem was solved for all quotient surface singularities after translating it
into a topological one (cf.\ \cite{FdB12}); 
notably, a similar approach was later applied to tackle the general case in \cite{FdBPP12}. 
In all of the above works the ground field is assumed 
to have characteristic zero. This assumption is especially relevant for the proof of the $E_8$
case which uses transcendental methods in an essential way. 

A purely algebraic proof of the Nash problem for surfaces was given in \cite{dFD16}. 
Both proofs from \cite{FdBPP12,dFD16} rely on wedges and the curve selection lemma.
In fact, as remarked by Lejeune-Jalabert and Reguera, the proof of \cite{dFD16}
solves an \emph{a priori} harder problem, known as the \emph{wedge lifting problem}, which goes back to \cite{LJ80}
and is known to imply the Nash problem \cite{Reg04}.  
By contrast, it was shown in \cite{Reg12} that the wedge problem fails in positive characteristics;
the examples are given by an $E_8$ singularity in characteristics $2,3,5$. 
Reguera's counterexamples to the wedge lifting problem easily generalizes to
all characteritics (see \cref{ss:Appendix}).

With \cref{t:Nash-DV} now settling the Nash problem for $E_8$ in all characteristics, we
conclude that these two problems---the Nash problem and the wedge lifting problem---are 
not equivalent in positive characteristic.

The proof in \cite{dFD16} breaks down in positive characteristics
due to the possible occurrence of wild ramification, which invalidate a key
estimate in the proof. 
With the approach via wedges failing in positive characteristics, 
we address \cref{t:Nash-DV} via a direct computation from the given equations. 
The idea is to construct the irreducible components in the space of arcs
from the irreducible components at the finite jet level, by going up one level at a time. 
The proof is inspired by the computations of the
irreducible components of the jets spaces through rational double points carried out 
in characteristic zero in \cite{Mou14}, though there are some differences in the way
the computations are organized here. 

In \cite{dFW24}, we constructed a map $\Psi_{X,m}$, defined for $m \gg 1$, from the set of irreducible components
of the space of arcs through a singularity and the set of irreducible components
of the space of $m$-jets through the singularity. There, we used results on the Nash problem
to show that this map is surjective 
for a certain class of singularities of characteristic zero that generalize the notion
of Du Val singularities in higher dimensions. 
Here, looking at two-dimensional rational double points of arbitrary characteristic,
we follow an opposite approach, where we first compute the irreducible components at the jet level, 
and then show that map $\Psi_{X,m}$ is surjective, deducing from there the validity of the Nash problem. 

A crucial part of the proof is to
distinguish the components at the arc level 
(which is equivalent, in this setting, 
to showing that $\Psi_{X,m}$ is surjective). This in principle
requires controlling infinitely many equations. 
We make this possible by carefully organizing our computations, 
taking radicals at each step to simplify the equations of the sets.

Our computations apply mostly uniformly to all characteristics, though
some minor adjustments are needed in small characteristics 
(between characteristics 2 and 3 for $E_6$ and $E_7$,  
and between characteristics 2, 3 and 5 for $E_8$). 
The fact that the computations are sensitive
to small characteristics can be seen as a manifestation of
the subtlety of the problem in that setting, as one should expect in view of Reguera's 
aforementioned counterexamples to the wedge lifting problem. 
In the cases where different isomorphism types exist
in small characteristics, the same computations carry through for all isomorphism types.

\subsection*{Acknowledgements}

We wish to thank Javier Fernandez de Bobadilla and Roi Docampo for useful discussions. 
We are grateful to the referee for helpful suggestions.

\section{Jet and arc components through singularities}

Let $k$ be an algebraically closed field.
For a $k$-algebra $R$, let $R_m$ be the universal Hasse--Schmidt $k$-algebra of order $m$ of $R$ (cf.\ \cite{Voj07}). 
For any $f \in R$ and $i \le m$, we denote by $f_i \in R_m$ the $i$-th
Hasse--Schmidt derivative of $f$. 
For instance, if $R = k[x,y,z]$ then $R_m = k[x_i,y_i,z_i \mid 0 \le i \le m]$, 
and if $R = k[x,y,z]/(f)$ for some $f \in k[x,y,z]$ then 
$R_m = k[x_i,y_i,z_i \mid 0 \le i \le m]/(f_j \mid 0 \le j \le m)$.

Let $X$ be a $k$-scheme and $m \in \N$. 
We denote by $X_m$ and $X_\infty$ the $m$-jet scheme and arc space
of $X$, respectively (cf.\ \cite{EM09}). These are schemes over $k$
whose $K$-valued points, for any field extension $K/k$, are 
jets $\g \colon \Spec K[t]/(t^{m+1}) \to X$ and arcs
$\a \colon \Spec K[[t]] \to X$ with coefficients in $K$. 
If $X = \Spec R$ for a $k$-algebra $R$, then we have $X_m = \Spec R_m$. 
We denote by $\p_m = \p_{X,m} \colon X_m \to X$ and $\p = \p_X \colon X_\infty \to X$ the natural projections
mapping a jet $\g(t)$ and an arc $\a(t)$ to their base point $\g(0), \a(0) \in X$, respectively. 
For a closed set $\S \in X$, we denote by $X_m^\S := \p_m^{-1}(\S)_\red$
and $X_\infty^\S := \p^{-1}(\S)_\red$ the sets of $m$-jets and of arcs through $\S \subset X$. 

Assume now that $X$ is a variety and $\S = \Sing X$ is its singular locus. 
Then $X_m$ is a scheme of finite type over $k$, and therefore
$X_m^\S$, as a closed subscheme of a scheme of finite type, has finitely many irreducible components. 
We call them the \emph{$m$-jet components} through $\S \subset X$. 
In general, however, $X_\infty$ is not of finite type, and in fact it is not Noetherian. 
Nonetheless, it is still the case that $X_\infty^\S$ has finitely many irreducible
components as long as $X$ admits a resolution of singularities \cite[Corollary~8.8]{dF18}.
We call them the \emph{arc components} through $\S \subset X$. 

If either $\chr k = 0$ or if $X$ has isolated singularities
and admits a resolution of singularities $f \colon Y \to X$, 
then every irreducible component $C$ of $X_\infty^\S$ is 
equal to the closure of a set of the form $f_\infty(\p_Y^{-1}(E))$
for some prime divisor $E$ on $Y$ with $f(E) \subset \S$;
here $f_\infty \colon Y_\infty \to X_\infty$ is the natural map
defined by $f_\infty(\a) := f\o\a$. In general, the same is true if we restrict our attention to arc components 
that are not fully contained in $\S_\infty$; we refer to these as the \emph{non-degenerate} components.
This correspondence associating $\ord_E$ to $C$ defines an injective map 
\[
\cN_X \colon \{\text{non-degenerate arc components}\} 
\to \{\text{essential divisorial valuations}\},
\]
called the \emph{Nash map}, from the set of non-degenerate arc components through the singularities of $X$
to the set of a special class of valuations on $X$, 
called \emph{essential divisorial valuations} \cite{Nas95,IK03}. 
When $X$ is a surface, these are precisely
the valuations defined by the exceptional divisors on the minimal resolution. 

At its core, the Nash problem aims to understand the image of this map. 
If $(X,0)$ is a surface singularity, then the question is whether 
the image of the Nash map $\cN_X$ is the set of divisorial
valuations defined by the exceptional divisors on the minimal resolution of singularities
of $(X,0)$.
As we already discussed in the introduction, \cite{FdBPP12} solves this problem
in characteristic zero. 

Arc components also relate to jet components. This was studied in \cite{dFW24}, where an
injective map
\[
\Psi_{X,m} \colon \{\text{non-degenerate arc components}\} \to \{\text{$m$-jet components}\}
\]
from the set of non-degenerate arc components
through the singular locus $\S$ of $X$ 
 to the set of $m$-jet components through $\S$ is defined for all $m \gg 1$. 
This map sends a non-degenerate arc component $C$ to the unique $m$-jet component $D$
containing the image of $C$ under the truncation map $X_\infty \to X_m$ 
\cite[Theorem 4.1]{dFW24}. The characteristic is there assumed to be
zero, but the construction of $\Psi_{X,m}$ works in all characteristics. The 
definition of $\Psi_{X,m}$ extends to all $m$ as a multivalued function. 

The function $\Psi_{X,m}$ is related to the directed graph $\G_X$ 
determined by all jet components as $m$ varies. This graph was studied in 
\cite{Mou11,Mou14,Mou17,CM21}, and is defined as follows:
vertices $v$ correspond to jet components through $\S \subset X$
and are labelled by the order $m$ of the jets, 
and an edge is drawn from a vertex $v$ of order $m$ to a vertex $v'$ of order $m+1$ whenever
the component corresponding to $v'$ maps into the one corresponding to $v$ under the
truncation map $X_{m+1} \to X_m$. 
We refer to $\G_X$ as the \emph{graph of jet components} through the singular
locus $\S$ of $X$. 

It is natural to connect the Nash problem to properties of the graph $\G_X$. 
With this in mind, we introduce some terminology. 
A \emph{branch} of $\G_X$ stemming from a vertex $v$ is the directed subgraph of $\G_X$ obtained
by removing all vertices that are not reachable by $v$. 
A \emph{directed chain} in $\G_X$ is a connected directed subgraph of $\G_X$ such that
every vertex in it has at most one edge stemming from it. 
A branch is said to be \emph{simple} if it is a directed chain in $\G_X$.
A subgraph of $\G_X$ is \emph{unbounded} if it contains vertices of arbitrarily large order.
We regard unbounded subgraphs of $\G_X$ up to equivalence, 
where any two such subgraphs are said to be equivalent if they agree 
in large enough order (i.e., after removing all vertices of oder $< m$ 
for some $m$). 

In view of the injectivity and well-definedness of $\Psi_{X,m}$ for $m \gg 1$, 
each arc component through $\S \subset X$ determines a unique (up to equivalence) unbounded directed chain
in $\G_X$, hence we have an injective map
\[
\Psi_X \colon \{\text{non-degenerate arc components}\} \to \{\text{unbounded directed chains}\}
\]
which sends an arc component $C$ to the unbounded directed chain in $\G_X$
whose vertex of order $m$, for $m \gg 1$, corresponds to the $m$-jet component given by $\Psi_{X,m}(C)$. 

In some cases (e.g., when $\Psi_{X,m}$ is surjective for $m \gg 1$), $\Psi_X$ is surjective and
every unbounded directed chain in $\G_X$ is equivalent to an unbounded simple branch, 
but the latter property fails in general (e.g., for non-canonical toric surface singularities \cite{Mou17}) 
and it is unclear whether the first should hold in general.

These properties are closely related
to the Nash problem. Knowing the image of the Nash map can help determine when 
these properties hold (see, e.g., \cite{dFW24}). 
Conversely, computing jet components and understanding the image 
of $\Psi_X$ provides a way to determine the image of the Nash map. 
The latter is our approach to prove \cref{t:Nash-DV}. 
We summarize our main result in the next statement, which includes \cref{t:Nash-DV}.

\begin{theorem}
\label{t:main-DV}
Let $(X,0)$ be a two-dimensional rational double point defined over an algebraically closed field $k$. Then:
\begin{enumerate}
\item
The map $\cN_X$ from the set of non-degenerate arc components through $0 \in X$
to the set of valuations defined by the exceptional divisors on the minimal resolution
of $X$ is surjective, hence a bijection. 
\item
The map $\Psi_{X,m}$ from the set of arc-components to the set of $m$-jet components 
through $0 \in X$ is surjective, hence a bijection, 
for every $m \gg 1$ (and for every $m$ as a multivalued function).
\item
The map $\Psi_X$ from the set of arc-components through $0 \in X$ to the set 
of equivalent classes of unbounded directed chains in $\G_X$ is surjective, hence a bijection, 
for every $m$. Furthermore, 
every unbounded directed chain is equivalent to an unbounded simple branch of $\G_X$. 
\end{enumerate}
\end{theorem}

The proof of \cref{t:main-DV} is contained in the following sections,
which are organized depending on the type of singularity. 
We summarize in the next section some general principle that will guide our computations. 
We are hoping that a similar approach may be applied
in the future to study the Nash maps in 
other situations as well, including higher dimensional cases.

\begin{notation}
The following notation will be used in the proof of \cref{t:main-DV}.
Let $I$ and $J$ be ordered index sets. 
For a set of variables $\{u_i \mid i \in I\}$ and $m\in I$, 
we use the short-hand notation $u_{\le m}$ to denote the set $\{u_i \mid i \le m\}$
and $u_{> m}$ to denote $\{u_i \mid i > m\}$. 
For any $j \in I$ with $j < m$, we denote the set
$\{u_i \mid i \le m, i \ne j\}$ by $u_{\le m} \setminus u_j$. 
In a polynomial ring $R = A[u_i,v,w_j \mid (i,j) \in I \times J]$, 
given $(m,n) \in I \times J$ we write
$V(u_{\le m})_v \mid_{w_j = P_j,\;j \ge n}$
to denote a subset of $\Spec R$ obtained by first setting $u_i = 0$ for all $i \le m$, 
then localizing at $v$, and finally eliminating the variables $w_j$, for $j \ge n$, 
using equations $w_j = P_j$ where $P_j \in A[u_{> m},v,w_{> j}]_v$. 
Finally, given a nonzero divisor $f$ and an ideal $\fa$ in a ring $R$, 
we denote by $D(f) := \Spec R_f$ the localization at $f$, and
use the symbols $Z(\fa)$ and $V(\fa)$ to denote the subscheme and the
(reduced) zero locus of $\fa$ in $\Spec R$. 
\end{notation}

\section{Computational strategy}

Let $f \in k[x, y, z]$ be a polynomial, and let $f_i$ denote its $i$-th Hasse–Schmidt derivatives. 
We summarize two key observations that are implicitly used
in our analysis of the fiber $X_\infty^0$ when $X = \{f=0\} \subset \A^3$
is a surface with a rational double point at $0$. Similar properties hold for polynomials in any number of variables.

\begin{proposition}
\label{prop:modulo-rule}
Let 
\[
f = \sum_{(a,b,c) \in I} r_{a,b,c} \. x^ay^bz^c
\]
with $r_{a,b,c} \in k$, 
and assume that the characteristic is either zero or larger
than any of the exponents $a,b,c$ appearing in $f$. 
Suppose that 
$f_0, f_1, \ldots, f_{n-1} \in (x_{\le i-1}, y_{\le j-1}, z_{\le h-1})$.
Then the next derivative satisfies
\[
f_n \equiv f|_J(x_i, y_j, z_h) \mod (x_{\le i-1}, y_{\le j-1}, z_{\le h-1}),
\]
for some (possibly empty) $J \subset I$, where
\[
f|_J := \sum_{(a,b,c) \in J} r_{a,b,c} \. x^ay^bz^c.
\]
That is, $f_n$, modulo lower-order derivatives, is congruent to $f|_J$ 
evaluated at the variables $x_i, y_j, z_h$. Importantly, each variable appears only in one specific order of its derivative; for instance, only 
$x_i$ appears and not $x_{i+1}, x_{i+2},\ldots, x_n.$
\end{proposition}

\begin{proof}
Suppose, for contradiction, that after reducing modulo $(x_{\le i-1}, y_{\le j-1}, z_{\le h-1})$, 
the derivative $f_n$ remains nonzero and contains a monomial involving a variable of 
strictly higher order than those in $\{x_i, y_j, z_h\}$.
Without loss of generality, assume there is a variable  $x_{i+\ell}$ for some $\ell \ge 1$, 
appearing in a monomial of $f_n$ of the form
$$
u \cdot x_{i + \ell}^\alpha \cdot x_i^\beta \cdot M,
$$
where $u \in k$, $\alpha > 0$, $\beta\geq0 $, and $M$ is a monomial in 
$k[x_{i'}, y_{j'}, z_{h'} \mid i' \ge i+\ell+1, j' \ge j, h' \ge h]$. 
In other words, $\ell$ is chosen to be the smallest positive integer such that $x_{i+\ell}$ appears in 
this monomial of $f_n$.
Then $f_{n-\alpha \ell}$ contains a monomial of the form
\[
v\. x_i^{\alpha+\beta} \cdot M,
\]
where $v \in k$ differs from $u$ only by factors given by 
binomial coefficients. Specifically, we have $u = v \. \binom{\a+\b}{\a}$.
It follows that if $u \ne 0$ then $v \ne 0$. 
This gives a contradiction to our assumption that
$f_{n-\alpha \ell} \in (x_{\le i-1}, y_{\le j-1}, z_{\le h-1})$.
\end{proof}

\begin{proposition}
\label{prop:modulo-smooth}
With the same assumptions and notation as in Proposition~\ref{prop:modulo-rule}, 
suppose that $f|_J \ne 0$ and let $g := f|_J(x_i,y_j,z_h)$. 
Consider the scheme $Z := Z(x_{\le i-1}, y_{\le j-1}, z_{\le h-1}, g)$,
and let $Z_\sm$ denote the smooth locus of $Z$, i.e.,
\[
Z_\sm = Z \cap \Big(D \Big( \frac{\partial g}{\partial x_i} \Big) \cup 
D\Big( \frac{\partial g}{\partial y_j} \Big) \cup 
D\Big( \frac{\partial g}{\partial z_k} \Big) \Big).
\]
Then on each of the open sets 
$Z \cap D \big( \frac{\partial g}{\partial x_i} \big)$, 
$Z \cap D\big( \frac{\partial g}{\partial y_j} \big)$, and 
$Z \cap D\big( \frac{\partial g}{\partial z_k} \big)$, 
we can use the remaining equations $f_{n+\ell} = 0$, 
for $\ell \ge 0$, to eliminate the variables
$x_{i+\ell}$, $y_{j+\ell}$, or $z_{h+\ell}$, respectively. 
\end{proposition}

\begin{proof}
For any $\ell \ge 1$, we observe that, modulo 
$(x_{\le i-1}, y_{\le j-1}, z_{\le h-1})$, 
the $(n+\ell)$-th Hasse--Schmidt derivative of $f$ has the form
\begin{multline*}
f_{n+\ell} \equiv 
\frac{\partial g}{\partial x_i}(x_i, y_j, z_h) \cdot x_{i+\ell} +
\frac{\partial g}{\partial y_j}(x_i, y_j, z_h) \cdot y_{j+\ell} +
\frac{\partial g}{\partial z_k}(x_i, y_j, z_h) \cdot z_{h+\ell} \\
+ G_{n,\ell}(x_{i'}, y_{j'}, z_{h'}) \mod (x_{\le i-1}, y_{\le j-1}, z_{\le k-1}),
\end{multline*}
where 
$G_{n,\ell} \in k[x_{i'},y_{j'},z_{h'} \mid i\le i' < i+\ell, j\le j' < j+\ell, k\le k' < k+\ell]$.
The result easily follows from this. 
\end{proof}

\begin{remark}
Implicit in \cref{prop:modulo-smooth} is the assumption that $Z_\sm$ is nonempty, 
as otherwise the statement is vacuous. 
If $Z$ is generically smooth, then \cref{prop:modulo-smooth}
provides a way of computing the inverse image of its smooth locus in $X_\infty$. 
The proposition does not say anything, however, if these schemes are
non-reduced. In our setting, when this occurs in the proof of \cref{t:main-DV}, 
then $g$ typically consists of a single monomial. Those situations are treated by
setting each variable appearing in the monomial equal to zero, 
and go one step up in the jet schemes to continue the computations. 
\end{remark}

\section{$A_n$ singularities}

In this section we verify \cref{t:main-DV} for $A_n$ singularities.
The computation of jet and arc components in this case is straightforward
and completely characteristic free. We include it here for completeness and
to gives an idea in this simple case of how the argument will be organized in the other cases. 

Let $(X,0)$ be an $A_n$ singularity, with $n \ge 1$.
We may assume that $X$ is defined in $\A^3$ by 
\[
f(x,y,z) = z^{n+1} + xy.
\]

\setcounter{step}{0}
\begin{step}
For $i = 1,\dots n-1$, we set 
$V_i := V(x_{\le i-1},y_0,z_0) \subset \A^3_\infty$.
Inductively, we see that $V_i \subset V(f_0,\dots,f_{i})$.
We define 
$W_i := V_i \cap V(f_{i+1})$. 
Note that
\[
f_{i+1} \equiv x_iy_1   \quad \mod (x_{\le i-1},y_0,z_0).
\]
We stratify 
$W_i = (W_i \cap D(x_{i})) \sqcup (W_i \cap V(x_{i}))$.
On $W_i \cap D(x_{i})$, we have $y_1 = 0$, and 
we use the equations $f_m = 0$, for $m \ge i+1$,
and the fact that $x_{i}$ is a unit on this set, to eliminate the variables $y_{m-i}$. 
The first $n-i$ of these equations give $y_1 = \dots = y_{n-i} = 0$, and we obtain the set
\[
C_i^\o = V(x_{\le i-1},y_{\le n-i},z_0)_{x_{i}}|_{y_{m-i} = (x_i)^{-1}P_{i,m},\; m \ge n+1}
\]
where the polynomials $P_{i,m}$ do not depend on $y_j$ for any $j \ge m-i$.
We denote by $C_i \subset X_\infty$ the closure of $C_i^\o$. 

We are left with 
$W_i \cap V(x_{i}) = V(x_{\le i},y_0,z_0)$.
If $i < n-1$, then this is precisely $V_{i+1}$, and we apply induction. 
If $i = n-1$, then we end up with 
$V_n := W_{n-1} \cap V(x_{n-1}) = V(x_{\le n-1},y_0,z_0)$.

Let $W_n := V_n \cap V(f_{n+1})$.
We have
\[
f_{n+1} \equiv z_1^{n+1} + x_ny_1   \quad \mod (x_{\le n-1},y_0,z_0).
\]
We cover $W_n$ with $D(x_n) \cup D(y_1) \cup V(x_n,y_1)$. 
On $W_n \cap D(x_n)$, we use the equations $f_{n+j} = 0$ for $j \ge 1$ to eliminate $y_j$.
This produces $C_n^\o$.  
On $W_n \cap D(y_1)$, we use the remaining equations to solve for $x_j$ for $j \ge n$.
The resulting set is irreducible and agree with $C_n^\o$ on $D(x_n) \cap D(y_1)$. 
Finally, on $W_n \cap V(x_n,y_1)$ we automatically have $z_1 = 0$, 
and we end up with a set contained in the closure of $C_n^\o$. 
This shows that $W_n \cap V(f_m \mid m \ge n+2)$ is irreducible and agrees with the closure $C_n$
of $C_n^\o$. 
\end{step}

\begin{step}
To conclude the proof, it suffices to note that the sets $C_1,\dots,C_n$ are irreducible, they 
fill up the fiber $X_\infty^0$, and there are no containments among any two of them, namely, 
$C_i \not\subset C_j$ whenever $i \ne j$. The first two assertions are clear from the construction, and
the last assertion follows immediately by comparing the defining equations of the sets $C_i^\o$.
\end{step}

\section{$D_{n}$ singularities}

In this section we look at $D_n$ singularities.
We give the proof of \cref{t:main-DV} in the even case $D_{2n}$. The proof for $D_{2n+1}$
singularities is similar and is omitted.

Let $(X,0)$ be a $D_{2n}$ singularity, with $n \ge 2$. 
We may assume that $X$ given in $\A^3$ by
\[
f(x,y,z) = z^2 + x^2y + xy^n + h(x,y,z) = 0,
\]
where $h \in \{0, xy^{n-r}z \mid 1 \le r \le n-2\}$ if the characteristic is 2, and $h = 0$ otherwise. 

The proof goes by 
constructing the locally closed sets 
$C_1^\o,\dots,C_{2n-2}^\o,C_{n-1}'^\o,C_{n-1}''^\o$
using the equations of $X_\infty^0$, and prove that their closures
$C_1,\dots,C_{2n-2},C_{n-1}',C_{n-1}''$ are the irreducible components of $X_\infty^0$. 

The construction breaks into four steps. First, we compute the
first $n-2$ sets $C_1^\o,\dots,C_{n-2}^\o$ by looking at the jet
schemes up to order $2n-3$. 
Next we compute $C_{n-1}^\o$, $C_{n-1}'^\o$ and $C_{n-1}''^\o$ from the jet scheme at level $2n-1$. 
We then compute $C_{n}^\o, \dots, C_{2n-3}^\o$ going up on higher levels,
until we reach level $4n-2$, which produces the last component $C_{2n-2}$. 
By construction, these sets are irreducible and cover all of $X_\infty^0$.
The last step of the proof will be to prove 
that there are no inclusions among these sets.

\setcounter{step}{0}
\begin{step}
For $i = 1,\dots n-2$, we set 
\[
V_i := V(x_{\le i-1},y_0,z_{\le i-1}) \subset \A^3_\infty.
\]
Inductively, we see that $V_i \subset V(f_0,\dots,f_{2i-1})$.
We cut $V_i$ with the next equations $f_{2i} = f_{2i+1} = 0$, and define
\[
W_i := V_i \cap V(f_{2i},f_{2i+1}).
\]
We compute
\[
\begin{cases}
f_{2i} \equiv z_i^2   &\mod (x_{\le i-1},y_0,z_{\le i-1}), \\
f_{2i+1} \equiv x_{i}^2y_1 &\mod (x_{\le i-1},y_0,z_{\le i}).
\end{cases}
\]
This implies that $W_i = V_i \cap V(z_i,x_{i}y_1)$.
We stratify $W_i = (W_i \cap D(x_{i})) \sqcup (W_i \cap V(x_{i}))$.
On $W_i \cap D(x_{i})$, we have $y_1 = 0$. 
For $m \ge 2i+2$, we have
\[
f_m \equiv x_1^2 y_{m-2} + G_{i,m} \mod (x_{\le i-1},y_{\le 1},z_{\le i})
\]
where $G_{i,m}$ does not depend on $y_{j-2}$ for any $j \ge m$. Therefore, 
we use the remaining equations $f_m = 0$, for $m \ge 2i+2$, 
and the fact that $x_{i}$ is a unit on this set, to eliminate the variables $y_{m-2i}$. 
This gives the set 
\[
C_i^\o = V(x_{\le i-1},y_{\le 1},z_{\le i})_{x_{i}}
|_{y_{m-2i} = (x_i^2)^{-1}P_{i,m},\; m \ge 2i+2}
\]
where $P_{i,m} = - g_{i,m}$.
We are left with
\[
W_i \cap V(x_{i}) = V(x_{\le i},y_0,z_{\le i}).
\]
If $i < n-2$, then this is precisely $V_{i+1}$, and we apply induction.
\end{step}

\begin{step}
The previous step gives us the sets $C_1^\o,\dots, C_{n-2}^\o$, 
and leaves us in the end with the set
\[
V_{n-1} := V(x_{\le n-2},y_0,z_{\le n-2}).
\]
Observing that $V_{n-1} \subset V(f_0,\dots,f_{2n-3})$, we define
\[
W_{n-1} := V_{n-1} \cap V(f_{2n-2},f_{2n-1}).
\]
We have
\[
\begin{cases}
f_{2n-2} \equiv z_{n-1}^2   &\mod (x_{\le n-2},y_0,z_{\le n-2}), \\
f_{2n-1} \equiv x_{n-1}^2 y_1 + x_{n-1} y_1^n &\mod (x_{\le n-2},y_0,z_{\le n-1}).
\end{cases}
\]
hence
$W_{n-1} = V_{n-1} \cap V(z_{n-1}, x_{n-1}^2 y_1 + x_{n-1} y_1^n)$.
We stratify $W_{n-1}$ using the stratification
\[
D(x_{n-1}y_1) \sqcup (D(x_{n-1}) \cap V(y_1)) \sqcup (V(x_{n-1}) \cap D(y_1)) \sqcup V(x_{n-1},y_1).
\]
This stratification reflects a decomposition of $W_{n-1}$ into
three irreducible components, cut out, respectively, by the three
equations $x_{n-1} = 0$, $y_1 = 0$, and $x_{n-1} + y_1^n = 0$, 
which we analyze next.

We first look at $W_{n-1} \cap D(x_{n-1}) \cap V(y_1)$. 
For $m \ge 2n$, we have 
\[
f_m \equiv x_{n-1}^2 y_{m-2n + 2} + g_{n-1,m} \mod (x_{\le n-2},y_{\le 1},z_{\le n-1}),
\]
where $g_{n-1,m}$ does not depend on $y_{j-2n+2}$ for any $j \ge m$. 
Thus we can
solve for $y_{m-2n+2}$ using the remaining equations $f_m = 0$ for $m \ge 2n$. This produces 
\[
C_{n-1}^\o = V(x_{\le n-2},y_{\le 1},z_{\le n-1})_{x_{n-1}}
|_{y_{m-2n+2} = (x_{n-1}^2)^{-1}P_{n-1,m},\; m \ge 2n}
\]
where $P_{n-1,m} = - g_{n-1,m}$.

Next, we look at $W_{n-1} \cap D(y_1) \cap V(x_{n-1})$. 
For $m \ge 2n$, we have 
\[
f_m \equiv x_{m-n} y_1^n + G_{n-1,m}' \mod (x_{\le n-1},y_{\le 0},z_{\le n-1}),
\]
where $G_{n-1,m}'$ does not depend on $x_{j-n}$ for any $j \ge m$. 
Thus we can
solve for $x_{m-n}$ using the remaining equations $f_m = 0$ for $m \ge 2n$, obtaining 
\[
C_{n-1}'^\o = V(x_{\le n-1},y_0,z_{\le n-1})_{y_1}
|_{x_{m-n} = (y_1^n)^{-1}P_{n-1,m}',\; m \ge 2n},
\]
where $P_{n-1,m}' = - G_{n-1,m}'$.

We now analyze at the stratum 
$W_{n-1} \cap D(x_{n-1}y_1)$. Note that on this set we have $x_{n-1}+y_1^{n-1} \equiv 0$, 
a relation that is used without warning in the next computation.
For $m \ge 2n$, we have
\[
f_m \equiv x_{n-1}x_{m-n}y_1 + G_{n-1,m}'' \mod (x_{\le n-2},y_{\le 0},z_{\le n-1})
\]
where $G_{n-1,m}''$ does not depend on $x_{j-n}$ for any $j \ge m$.
We use the equations $f_m = 0$ to solve 
for $x_{m-n}$ for $m \ge 2n$, by inverting $2 x_{n-1}y_1$. This yields
\[
C_{n-1}''^\o = V(x_{\le n-2},y_0,z_{\le n-1},x_{n-1}+y_1^{n-1})_{x_{n-1}y_1}
|_{x_{m-n} = (x_{n-1}y_1)^{-1}P_{n-1,m}'',\; m \ge 2n}
\]
where
$P_{n-1,m}'' = - G_{n-1.m}''$.

We are left with
$W_{n-1} \cap V(x_{n-1},y_1)$. Note that this is equal to $V(x_{\le n-1},y_{\le 1},z_{\le n-1})$.
We cut with with the next equation $f_{2n} = 0$, and define
\[
V_n := W_{n-1} \cap V(x_{n-1},y_1,f_{2n}).
\]
Observing that 
\[
f_{2n} \equiv z_n^2   \quad \mod (x_{\le n-1},y_0,z_{\le n-1}),
\]
we compute
$V_n = V(x_{\le n-1},y_{\le 1},z_{\le n})$.
\end{step}

\begin{step}
For $i = n,\dots 2n-3$, consider the set 
\[
V_i := V(x_{\le i-1},y_{\le 1},z_{\le i}).
\]
Note that, for $i = n$, the definition of this set agrees with the set $V_n$
defined at the end of the previous step. 
Inductively, we see that $V_i \subset V(f_0,\dots,f_{2i})$.
We define
\[
W_i := V_i \cap V(f_{2i+1},f_{2i+2}).
\] 
Note that
\[
\begin{cases}
f_{2i+1} \equiv 0 &\mod (x_{\le i-1},y_{\le 1},z_{\le i}),\\
f_{2i+2} \equiv z_{i+1}^2 + x_{i}^2 y_2 &\mod (x_{\le i-1},y_{\le 1},z_{\le i}).
\end{cases}
\]
Therefore
$W_i = V_i \cap V(z_{i+1}^2 + x_{i}^2 y_2)$.
We stratify 
$W_i = (W_i \cap D(x_{i})) \sqcup (W_i \cap V(x_{i}))$.
On $W_i \cap D(x_{i})$, we have
\[
f_m \equiv x_i^2 y_{m-2i} + G_{i,m} \mod (x_{\le i-1},y_{\le 1},z_{\le i})
\]
for $m \ge 2i+2$, where $G_{i,m}$ does not depend on $y_{j-2i}$ for any $j \ge m$.
We can eliminate the variables $y_{m-2i}$, obtaining
\[
C_i^\o = V(x_{\le i-1},y_{\le 1},z_{\le i})_{x_{i}}
|_{y_{m-2i} = (x_i^2)^{-1}P_{i,m},\; m \ge 2i+2}
\]
where $P_{i,m} = - G_{i,m}$.
We are left with $W_i \cap V(x_{i})$. On this set, we have $z_{i+1} = 0$, hence
\[
W_i \cap V(x_{i}) = V(x_{\le i},y_{\le 1},z_{\le i+1}).
\]
If $i < 2n-3$, then this is precisely $V_{i+1}$, and we apply induction.
\end{step}

\begin{step}
The previous steps have produced the strata $C_1^\o,\dots,C_{2n-3}^\o, C_{n-1}'^\o,C_{n-1}''^\o$, and
we are left with analyzing the set
\[
V_{2n-2} := W_{2n-3} \cap V(x_{2n-3},f_{4n-4}) = V(x_{\le 2n-3},y_{\le 1},z_{\le 2n-2}).
\]
So far, we have used the equations $f_m = 0$ for $m \le 4n-4$. 
We define
\[
W_{2n-2} := V_{2n-2} \cap V(f_{4n-3},f_{4n-2}).
\] 
We have
\[
\begin{cases}
f_{4n-3} \equiv 0 &\mod (x_{\le 2n-3},y_{\le 1},z_{\le 2n-2}),\\
f_{4n-2} \equiv z_{2n-1}^2 + x_{2n-2}^2y_2 + x_{2n-2}y_2^n & \mod (x_{\le 2n-3},y_{\le 1},z_{\le 2n-2}).
\end{cases}
\]
On $W_{2n-2} \cap D(x_{2n-2})$, 
we have
\[
f_m \equiv x_1^2 y_{m-2} + x_{m - 2n-1}y_1^n + G_{2n-3,m} \mod (x_{\le 2n-3},y_{\le 1},z_{\le 2n-2})
\]
for $m \ge 4n-1$, where $G_{2n-3,m}$ does not depend on $y_{j-2}$ nor $x_{j - 2n-1}$
for any $j \ge m$.

We cover $W_{2n-2}$ with $D(x_{2n-2}) \cup D(y_2) \cup V(x_{2n-2},y_2)$.
On $W_{2n-2} \cap D(x_{2n-2})$, we use the equations
$f_m = 0$ to eliminate the variables $y_{m-2}$ for $m \ge 4m-1$, obtaining
\[
C_{2n-2}^\o = V(x_{\le 2n-3},y_{\le 1},z_{\le 2n-2})_{x_{2n-2}}
|_{y_{m-2} = (y_{m-2})^{-1}P_{2n-3,m},\; m \ge 2i+2}
\] 
where $P_{2n-3,m} = - x_{m - 2n-1}y_1^n - G_{2n-3,m}$. 
On $W_{2n-2} \cap D(y_2)$, we eliminate the variables $x_{m-2n-1}$ 
using the same equations $f_m = 0$, but dividing this time by $y_2^n$; 
the resulting set agrees with $C_{2n-2}^\o$ on $D(x_{2n-2}) \cap D(y_2)$.
Finally, on $W_n \cap V(x_{2n-2},y_2)$ we 
end up with something contained in the closure of $C_{2n-2}^\o$.  

This shows that $W_{2n-2} \cap V(f_m \mid m \ge 4n-2)$ is irreducible and agrees with the closure
of $C_{2n-2}^\o$. This realizes the component $C_{2n-2}$.
\end{step}

\begin{step}
Let $C_i,C_{n-1}',C_{n-1}''$ denote the closure of the 
sets $C_i^\o, C_{n-1}'^\o,C_{n-1}''^\o$ constructed above. 
It is clear by construction that these sets are all irreducible
and fill up $X_\infty^0$. The last step is to show that there are no inclusions among them, 
hence $X_\infty^0$ has $2n$ irreducible components. 

By construction, we clearly have that $C_i \not\subset C_j$ whenever $i < j$, 
and $C_i \not\subset C_{n-1}' \cup C_{n-1}''$ for $i \le n-2$. 
Similarly, $C_{n-1} \not\subset C_{n-1}' \cup C_{n-1}''$, 
since $x_{n-1} \not\equiv 0$ on $C_{n-1}$
(showing that $C_{n-1} \not\subset C_{n-1}'$) and $x_{n-1} + y_1^{n-1} \not\equiv 0$
on $C_{n-1}$ (showing that $C_{n-1} \not\subset C_{n-1}''$).
Furthermore, we have $x_{n-1} \equiv 0$ and $x_{n-1} + y_1^{n-1} \not\equiv 0$ on $C_{n-1}'$
while $x_{n-1} \not\equiv 0$ and $x_{n-1} + y_1^{n-1} \equiv 0$ on $C_{n-1}''$, 
hence there are no inclusions among these two sets. 
Note also that $C_{n-1}',C_{n-1}'' \not\subset C_i$ for all $i$ since
$y_1$ is identically zero on $C_i$ and does not vanish identically on either 
$C_{n-1}'$ or $C_{n-1}''$. 

It remains to show that $C_j \not\subset C_i$ for $j > i$
and $C_i \not\subset C_{n-1}' \cup C_{n-1}''$ for $i \ge n$.
Regarding the first, observe that
\[
\begin{cases}
f_{2i + 2} \equiv x_i^2 y_2 &\mod (x_{\le j-1} \setminus x_i,y_{\le 1},z_{\le j}), \\
f_{2j + 2} \equiv x_j^2 y_2 + z_{j+1}^2 &\mod (x_{\le j-1},y_{\le 1},z_{\le j}). 
\end{cases}
\]
Keeping in mind that $x_i$ is a unit on $C_i^\o$, we see from the vanishing of $f_{2i+2}$ 
that $y_2 \equiv 0$ on $C_i \cap V(x_{\le j-1} \setminus x_i,y_{\le 1},z_{\le j})$
On the other hand, using that $x_j$ is a unit on $C_j^\o$ and $z_{j+1}$
is a free variable on this set, we see from the 
the vanishing of $f_{2j+2}$ that $y_2 \not\equiv 0$ on $C_j$. 
As $C_j \subset V(x_{\le j-1} \setminus x_i,y_{\le 1},z_{\le j})$, 
we conclude that $C_j \not\subset C_i$ whenever $j > i$. 

Finally, fix any $C_i$ with $i \ge n$. Note that
\[
f_{2i + 1} \equiv x_i^2 y_1 \mod (x_{\le i-1},y_0,z_{\le i}).
\]
As $y_1$ restricts to a unit on both $C_{n-1}'^\o$ and $C_{n-1}''^\o$, we 
deduce that $x_i \equiv 0$ on $(C_{n-1}' \cup C_{n-1}'') \cap V(x_{\le i-1},y_0,z_{\le i})$. 
On the other hand, $x_i \not\equiv 0$ on $C_i$, and 
since $C_i \subset V(x_{\le i-1},y_0,z_{\le i})$, we conclude that 
$C_i \not\subset C_{n-1}' \cup C_{n-1}''$ for $i \ge n$.
\end{step}

\section{$E_6$, $E_7$, and $E_8$ singularities}

In this section we look at the remaining exceptional cases. 
We present the proof of \cref{t:main-DV} for $E_8$.
The proofs of the other two cases are similar
and are left to the reader. 

We therefore assume that $X$ is given in $\A^3$ by
\[
f(x,y,z) = z^2 + x^3 + y^5 + h(x,y,z),
\]
where $h \in \{0, xy^3z, xy^2z, y^3z, xyz \}$ if $\chr k = 2$,  
$h \in \{0, x^2y^3, x^2y^2 \}$ if $\chr k = 3$,  
$h \in \{0, xy^4\}$ if $\chr k = 5$,
and $h = 0$ otherwise. 

We will construct locally closed sets $C_1^\o,\dots,C_8^\o$
whose closures $C_1,\dots,C_8$ are irreducible and fill up $X_\infty^0$. 
The last step of the proof will be to prove 
that there are no inclusions among these sets.

\setcounter{step}{0}
\begin{step}
We set $V_1 := V(x_0,y_0,z_0)$. 
The first nontrivial equations are
\[
\begin{cases}
f_2 \equiv z_1^2 &\mod (x_0,y_0,z_0), \\
f_3 \equiv x_1^3 &\mod (x_0,y_0,z_{\le 1}), \\
f_4 \equiv z_2^2 &\mod (x_{\le 1},y_0,z_{\le 1}), \\
f_5 \equiv y_1^5 &\mod (x_{\le 1},y_0,z_{\le 2}) \\
f_6 \equiv z_3^2 + x_2^3 &\mod (x_{\le 1},y_{\le 1},z_{\le 2}).
\end{cases}
\]
Let $W_1 := V_1 \cap V(f_2,\dots,f_6)$. On $W_1 \cap D(x_2)$, both $x_2$ and $z_3$ are units
(the fact that $z_3$ is a unit on $W_1 \cap D(x_2)$ follows from the equation $z_3^2 + x_2^3 = 0$, 
which holds on this set). For $m \ge 7$, we can write
\[
f_m \equiv 2z_3z_{m-3} + 3x_2^2x_{m-4} + G_{1,m} \mod (x_{\le 1},y_{\le 1},z_{\le 2}),
\]
where $G_{1,m}$ does not depend on 
$z_{j-3}$ or $x_{j-4}$ for any $j \ge m$. 

If the characteristic is different from 2, we can divide by $2z_3$ to solve 
for $z_{m-3}$ using the equation $f_m = 0$ for $m \ge 7$. 
If the characteristic is different from 3, we can divide by $3x_2^2$ and use the same
equations to solve for $x_{m-4}$ for $m \ge 7$. Either way, we obtain the irreducible
locally closed set
\[
C_1^\o = 
\begin{cases}
V(x_{\le 1},y_{\le 1},z_{\le 2}, z_3^2+x_2^3)_{x_2}
|_{z_{m-3} = (2z_3)^{-1} P_{1,m},\; m \ge 7} & (\chr k \ne 2) \\
V(x_{\le 1},y_{\le 1},z_{\le 2}, z_3^2+x_2^3)_{x_2}
|_{x_{m-4} = (3x_2^2)^{-1} Q_{1,m},\; m \ge 7} & (\chr k \ne 3)
\end{cases}
\]
where $P_{1,m} = - 3x_2^2x_{m-4} - G_{1,m}$ and
$Q_{1,m} = - 2z_3z_{m-3}  - G_{1,m}$.

Note that the two definitions of $C_1^\o$ given for $\chr k \ge 5$ agree.
The closure of $C_1^\o$ defines an irreducible set $C_1 \subset X_\infty^0$.
\end{step}

\begin{step}
On $W_1 \cap V(x_2)$ we have $z_2 = 0$, and this leaves us with 
\[
V_2 := W_1 \cap V(x_2) = V(x_{\le 2},y_{\le 1},z_{\le 3}).
\] 
Next, we have 
\[
\begin{cases}
f_7 \equiv 0 &\mod (x_{\le 2},y_{\le 1},z_{\le 3}), \\
f_8 \equiv z_4^2 &\mod (x_{\le 2},y_{\le 1},z_{\le 3}), \\
f_9 \equiv x_3^3 &\mod (x_{\le 2},y_{\le 1},z_{\le 4}), \\
f_{10} \equiv z_5^2 + y_2^5 &\mod (x_{\le 3},y_{\le 1},z_{\le 4}).
\end{cases}
\]
Let $W_2 := V_2 \cap V(f_7,\dots,f_{10})$. 
On $W_2 \cap D(y_2)$, both $y_2$ and $z_5$ are units. For $m \ge 11$, we write
\[
f_m \equiv 2z_5z_{m-5} + 5y_2^4y_{m-8} + G_{2,m} \mod (x_{\le 3},y_{\le 1},z_{\le 4}),
\]
where $G_{2,m}$ does not depend on 
$z_{j-5}$ or $y_{j-8}$ for any $j \ge m$. 

If $\chr k \ne 2$, we use the equations $f_m = 0$ for $m \ge 11$
to solve for $z_{m-5}$, and if $\chr k \ne 5$, we use the same equations
to solve for $y_{m-8}$. We obtain the set
\[
C_2^\o = 
\begin{cases}
V(x_{\le 3},y_{\le 1},z_{\le 4}, z_5^2+y_2^5)_{y_2}
|_{z_{m-5}= (2z_5)^{-1} P_{2,m},\; m \ge 11} & (\chr k \ne 2) \\
V(x_{\le 3},y_{\le 1},z_{\le 4}, z_5^2+y_2^5)_{y_2}
|_{y_{m-8}= (5y_2^4)^{-1} R_{2,m},\; m \ge 11} & (\chr k \ne 5)
\end{cases}
\]
where $P_{2,m} = - 5y_2^4y_{m-8} - G_{2,m}$ and
$R_{2,m} = - 2z_5z_{m-5} - G_{2,m}$.
The closure of $C_2^\o$ defines an irreducible set $C_2 \subset X_\infty^0$.
\end{step}

\begin{step}
On $W_2 \cap V(y_2)$ we have $z_5 = 0$, and this leaves us with 
\[
V_3 := W_2 \cap V(y_2) = V(x_{\le 3},y_{\le 2},z_{\le 5}).
\] 
Next, we look at  
\[
\begin{cases}
f_{11} \equiv 0 &\mod (x_{\le 3},y_{\le 2},z_{\le 5}), \\
f_{12} \equiv z_6^2 + x_4^3 &\mod (x_{\le 3},y_{\le 2},z_{\le 5}).
\end{cases}
\]
Let $W_3 := V_3 \cap V(f_{11},f_{12})$. 
On $W_3 \cap D(x_4)$, both $x_4$ and $z_6$ are units. 
For $m \ge 13$, we write
\[
f_m \equiv 2z_6z_{m-6} + 3x_4^2x_{m-8} + G_{3,m} \mod (x_{\le 3},y_{\le 2},z_{\le 5}),
\]
where $G_{3,m}$ does not depend on 
$z_{j-6}$ or $x_{j-8}$ for any $j \ge m$.
We use the equations $f_m = 0$, for $m \ge 13$, to solve for $z_{m-6}$ if 
$\chr k \ne 2$, and to solve for $x_{m-8}$ if $\chr k \ne 3$. We obtain the set
\[
C_3^\o = 
\begin{cases}
V(x_{\le 3},y_{\le 2},z_{\le 5}, z_6^2 + x_4^3)_{x_4}
|_{z_{m-6} = (2z_6)^{-1}P_{3,m},\; m \ge 13} & (\chr k \ne 2) \\
V(x_{\le 3},y_{\le 2},z_{\le 5}, z_6^2 + x_4^3)_{x_4}
|_{x_{m-8} = (3x_4^2)^{-1}Q_{3,m},\; m \ge 13} & (\chr k \ne 3)
\end{cases}
\]
where
$P_{3,m} = - 3x_4^2x_{m-8} - G_{3,m}$ and
$Q_{3,m} = - 2z_6z_{m-6} G_{3,m}$.
Let $C_3 \subset X_\infty^0$ be the closure of $C_3^\o$.
\end{step}

\begin{step}
On $W_3 \cap V(x_4)$ we have $z_6 = 0$, and this leaves us with 
\[
V_4 := W_3 \cap V(x_4) = V(x_{\le 4},y_{\le 2},z_{\le 6}).
\] 
Next, we look at  
\[
\begin{cases}
f_{13} \equiv 0 &\mod (x_{\le 4},y_{\le 2},z_{\le 6}), \\
f_{14} \equiv z_7^2 &\mod (x_{\le 4},y_{\le 2},z_{\le 6}), \\
f_{15} \equiv x_5^3 + y_3^5 &\mod (x_{\le 4},y_{\le 2},z_{\le 7}).
\end{cases}
\]
Let $W_4 := V_4 \cap V(f_{13},f_{14}, f_{15})$. 
On $W_4 \cap D(x_5)$, both $x_5$ and $y_3$ are units. 
For $m \ge 16$, we write
\[
f_m \equiv 3x_5^2x_{m-10} + 5y_3^4y_{m-12} + G_{4,m} \mod (x_{\le 4},y_{\le 2},z_{\le 7}),
\]
where $G_{4,m}$ does not depend on 
$x_{j-10}$ or $y_{j-12}$ for $j \ge m$. 
We use the equations $f_m = 0$, for $m \ge 16$, to solve for $x_{m-10}$ if 
$\chr k \ne 3$, and to solve for $y_{m-12}$ if $\chr k \ne 5$. We obtain the set
\[
C_4^\o = 
\begin{cases}
V(x_{\le 4},y_{\le 2},z_{\le 7}, x_5^3 + y_3^5)_{x_5}
|_{x_{m-10} = (3x_5^2)^{-1}Q_{4,m},\; m \ge 16} & (\chr k \ne 3) \\
V(x_{\le 4},y_{\le 2},z_{\le 7}, x_5^3 + y_3^5)_{x_5}
|_{y_{m-12} = (5y_3^4)^{-1}R_{4,m},\; m \ge 16} & (\chr k \ne 5)
\end{cases}
\]
where
$Q_{4,m} = - 5y_3^4y_{m-12} - G_{4,m}$ and
$R_{4,m} = - 3x_5^2x_{m-10} - G_{4,m}$.
Let $C_4 \subset X_\infty^0$ be the closure of $C_4^\o$.
\end{step}

\begin{step}
On $W_4 \cap V(x_5)$ we have $y_3 = 0$, and this leaves us with 
\[
V_5 := W_4 \cap V(x_5) = V(x_{\le 5},y_{\le 3},z_{\le 7}).
\] 
Next, we look at  
\[
\begin{cases}
f_{16} \equiv z_8^2 &\mod (x_{\le 5},y_{\le 3},z_{\le 7}), \\
f_{17} \equiv 0 &\mod (x_{\le 5},y_{\le 3},z_{\le 8}), \\
f_{18} \equiv z_9^2 + x_6^3 &\mod (x_{\le 5},y_{\le 3},z_{\le 8}).
\end{cases}
\]
Let $W_5 := V_5 \cap V(f_{16},f_{17}, f_{18})$. 
On $W_5 \cap D(x_6)$, we have units $z_9$ and $x_6$. 
For $m \ge 19$, we write
\[
f_m \equiv 2z_9z_{m-9} + 3x_6^2x_{m-12} + G_{5,m} \mod (x_{\le 4},y_{\le 2},z_{\le 7}),
\]
where $G_{5,m}$ does not depend on 
$z_{j-9}$ or $x_{j-12}$ for any $j \ge m$. 
We use the equations $f_m = 0$, for $m \ge 19$, to solve for $z_{m-9}$ if 
$\chr k \ne 2$ and for $x_{m-12}$ if $\chr k \ne 3$. We obtain the set
\[
C_5^\o = 
\begin{cases}
V(x_{\le 5},y_{\le 3},z_{\le 8}, z_9^2 + x_6^3)_{x_6}
|_{z_{m-9} = (2z_9)^{-1}P_{5,m},\; m \ge 19} & (\chr k \ne 2) \\
V(x_{\le 5},y_{\le 3},z_{\le 8}, z_9^2 + x_6^3)_{x_6}
|_{x_{m-12} = (3x_6^2)^{-1}Q_{5,m},\; m \ge 19} & (\chr k \ne 3)
\end{cases}
\]
where
$P_{5,m} = - 3x_6^2x_{m-12} - G_{5,m}$ and
$Q_{5,m} = - 2z_9z_{m-9} - G_{5,m}$.
Let $C_5 \subset X_\infty^0$ be the closure of $C_5^\o$.
\end{step}

\begin{step}
On $W_5 \cap V(x_6)$ we have $z_9 = 0$, and this leaves us with 
\[
V_6 := W_5 \cap V(x_6) = V(x_{\le 6},y_{\le 3},z_{\le 9}).
\] 
Next, we look at  
\[
\begin{cases}
f_{19} \equiv 0 &\mod (x_{\le 6},y_{\le 3},z_{\le 9}), \\
f_{20} \equiv z_{10}^2 + y_4^5 &\mod (x_{\le 6},y_{\le 3},z_{\le 9}).
\end{cases}
\]
Let $W_6 := V_6 \cap V(f_{19},f_{20})$. 
On $W_6 \cap D(y_4)$, the units are $z_{10}$ and $y_4$. 
For $m \ge 21$, we write
\[
f_m \equiv 2z_{10}z_{m-10} + 5y_4^4y_{m-16} + G_{6,m} \mod (x_{\le 6},y_{\le 3},z_{\le 9}),
\]
where $G_{6,m}$ does not depend on 
$z_{j-10}$ or $y_{j-16}$ for any $j \ge m$. 
We use the equations $f_m = 0$, for $m \ge 21$, to solve for $z_{m-10}$ if 
$\chr k \ne 2$ and for $y_{m-16}$ if $\chr k \ne 5$. We obtain the set
\[
C_6^\o = 
\begin{cases}
V(x_{\le 6},y_{\le 3},z_{\le 9}, z_{10}^2 + y_4^5)_{y_4}
|_{z_{m-10} = (2z_{10})^{-1}P_{6,m},\; m \ge 21} & (\chr k \ne 2) \\
V(x_{\le 6},y_{\le 3},z_{\le 9}, z_{10}^2 + y_4^5)_{y_4}
|_{y_{m-16} = (5y_4^4)^{-1}R_{6,m},\; m \ge 21} & (\chr k \ne 5)
\end{cases}
\]
where
$P_{6,m} = - 5y_4^4y_{m-16} - G_{6,m}$ and
$R_{6,m} = - 2z_{10}z_{m-10} - G_{6,m}$.
Let $C_6 \subset X_\infty^0$ be the closure of $C_6^\o$.
\end{step}

\begin{step}
On $W_6 \cap V(y_4)$ we have $z_{10} = 0$, and this leaves us with 
\[
V_7 := W_6 \cap V(y_4) = V(x_{\le 6},y_{\le 4},z_{\le 10}).
\] 
Next, we look at  
\[
\begin{cases}
f_{21} \equiv x_7^3 &\mod (x_{\le 6},y_{\le 4},z_{\le 10}), \\
f_{22} \equiv z_{11}^2 &\mod (x_{\le 7},y_{\le 4},z_{\le 10}), \\
f_{23} \equiv 0 &\mod (x_{\le 7},y_{\le 4},z_{\le 11}), \\
f_{24} \equiv z_{12}^2 + x_8^3 &\mod (x_{\le 7},y_{\le 4},z_{\le 11}).
\end{cases}
\]
Let $W_7 := V_7 \cap V(f_{21},\dots,f_{24})$. 
On $W_7 \cap D(x_8)$, 
both $z_{12}$ and $x_8$ are units. 
For $m \ge 25$, we write
\[
f_m \equiv 2z_{12}z_{m-12} + 3x_8^2x_{m-16} + G_{7,m} \mod (x_{\le 7},y_{\le 4},z_{\le 11}),
\]
where $G_{7,m}$ does not depend on 
$z_{j-12}$ or $x_{j-16}$ for any $j \ge m$. 
We use the equations $f_m = 0$, for $m \ge 25$, to solve for $z_{m-12}$ if 
$\chr k \ne 2$, and to solve for $x_{m-16}$ if $\chr k \ne 3$. We obtain the set
\[
C_7^\o = 
\begin{cases}
V(x_{\le 7},y_{\le 4},z_{\le 11}, z_{12}^2 + x_8^3)_{x_8}
|_{z_{m-12} = (2z_{12})^{-1}P_{7,m},\; m \ge 25} & (\chr k \ne 2) \\
V(x_{\le 7},y_{\le 4},z_{\le 11}, z_{12}^2 + x_8^3)_{x_8}
|_{x_{m-16} = (3x_8^2)^{-1}Q_{7,m},\; m \ge 25} & (\chr k \ne 3)
\end{cases}
\]
where
$P_{7,m} = - 3x_8^2x_{m-16} - G_{7,m}$ and
$Q_{7,m} = - 2z_{12}z_{m-12} - G_{7,m}$.
Let $C_7 \subset X_\infty^0$ be the closure of $C_7^\o$.
\end{step}

\begin{step}
On $W_7 \cap V(x_8)$ we have $z_{12} = 0$, and this leaves us with 
\[
V_8 := W_7 \cap V(x_8) = V(x_{\le 8},y_{\le 4},z_{\le 12}).
\] 
Finally, we look at
\[
\begin{cases}
f_{25} \equiv y_5^5 &\mod (x_{\le 8},y_{\le 4},z_{\le 12}), \\
f_{26} \equiv z_{13}^2 &\mod (x_{\le 8},y_{\le 5},z_{\le 12}), \\
f_{27} \equiv x_9^3 &\mod (x_{\le 8},y_{\le 5},z_{\le 13}), \\
f_{28} \equiv z_{14}^2 &\mod (x_{\le 9},y_{\le 5},z_{\le 13}), \\
f_{29} \equiv 0 &\mod (x_{\le 9},y_{\le 5},z_{\le 14}), \\
f_{30} \equiv z_{15}^2 + x_{10}^3 + y_6^5 &\mod (x_{\le 9},y_{\le 5},z_{\le 14}).
\end{cases}
\]
Let $W_8 := V_8 \cap V(f_{25},\dots,f_{30})$. 
This time we localize at a product to make all variables involved in the last equation units:
on $W_8 \cap D(z_{15}x_{10})$, all three variables $z_{15}$, $x_{10}$, and $y_6$ are units. 
For $m \ge 31$, we write
\[
f_m \equiv 2z_{15}z_{m-15} + 3x_{10}^2x_{m-20} + 5y_6^4y_{m-24} 
+ G_{8,m} \mod (x_{\le 9},y_{\le 5},z_{\le 14}),
\]
where $G_{8,m}$ does not depend on 
$z_{j-15}$, $x_{j-20}$ or $y_{j-24}$ for any $j \ge m$. 
We use the equations $f_m = 0$, for $m \ge 31$, to solve for $z_{m-15}$ if 
$\chr k \ne 2$, and to solve for $x_{m-20}$ if $\chr k \ne 3$. We obtain the set
\[
C_8^\o = 
\begin{cases}
V(x_{\le 9},y_{\le 5},z_{\le 14}, z_{15}^2 + x_{10}^3 + y_6^5)_{z_{15}}
|_{z_{m-15} = (2z_{15})^{-1}P_{8,m},\; m \ge 31} & (\chr k \ne 2) \\
V(x_{\le 9},y_{\le 5},z_{\le 14}, z_{15}^2 + x_{10}^3 + y_6^5)_{x_{10}}
|_{x_{m-20} = (3x_{10}^2)^{-1}Q_{8,m},\; m \ge 31} & (\chr k \ne 3)
\end{cases}
\]
where
$P_{8,m} = - 3x_{10}^2x_{m-20} - 5y_6^4y_{m-24} - G_{8,m}$ and
$Q_{8,m} = - 2z_{15}z_{m-15} - 5y_6^4y_{m-24} - G_{8,m}$.
Let $C_8 \subset X_\infty^0$ be the closure of $C_8^\o$.
\end{step}

\begin{step}
In the above steps, we have constructed closed irreducible sets $C_1,\dots,C_8 \subset X_\infty$. 
These sets fill up the whole fiber over the origin, hence 
\[
X_\infty^0 = C_1 \cup \dots \cup C_8.
\]
All we are left to show is that each $C_i$ is an irreducible component
of $X_\infty^0$. Equivalently, we need to show that there are no inclusions among any two of
them. 

It is clear from the construction (where we constructed the sets $C_i^\o$ by 
a series of localizations and subsequent restrictions to the complements)
that $C_i \not \subset C_j$ whenever $i < j$. 
The delicate part is to show that, similarly, $C_i \not \supset C_j$ for all $i < j$, hence there 
are no inclusions among any two of these sets. 

Here we explain why $C_1 \not\supset C_2$. The other cases are treated similarly and are left to the
reader.

First, assume that $\chr k \ne 2$. We look at one particular defining equation of
$C_1^\o$ in $D(x_2)$, namely, $z_5 = (2z_3)^{-1} P_{1,8}$. 
The reason we want to focus on this equation is that $z_5 \not\equiv 0$ on $C_2$. 
Rewriting the equation as $z_3z_5 = \frac 12 P_{1,8}$, we can extend it beyond $D(x_2)$. 
We look at this equation on the set $V(x_{\le 3},y_{\le 1},z_{\le 4} \setminus z_3)$.
The polynomial $P_{1,8}$ vanishes identically over this set, 
thus the given equation reduces to $z_3z_5 = 0$. 
Note that $C_1 \cap V(x_{\le 3},y_{\le 1},z_{\le 4} \setminus z_3)$ is an irreducible set and, by what we just said,
$z_3z_5$ vanishes identically on this set. As $z_3 \not \equiv 0$ 
on $C_1 \cap V(x_{\le 3},y_{\le 1},z_{\le 4} \setminus z_3)$,   
we conclude that $z_5 \equiv 0$ over there. 
However, $z_5 \not\equiv 0$ over $C_2$, 
since the latter was constructed as the closure of the set $C_2^\o$, 
which is contained in $D(z_5)$. Noting that $C_2 \subset V(x_{\le 3},y_{\le 1},z_{\le 4} \setminus z_3)$,
we conclude that $C_2 \not\subset C_1$. 

If $\chr k = 2$, then we do not have the equation $z_5 = (2z_3)^{-1} P_{1,8}$, 
but we can use instead the equation $x_{4} = (3x_2^2)^{-1} Q_{1,8}$, 
which also comes from setting $f_8 = 0$. 
This time we look at the equation on the set $V(x_{\le 3}\setminus x_2,y_{\le 1},z_{\le 4})$. 
The polynomial $Q_{1,8}$ vanishes identically on this set, 
hence the equation reduces to $x_2^2x_4 = 0$. 
By the same argument as in the previous case, we conclude that
$x_4$ vanishes along $C_1 \cap V(x_{\le 3}\setminus x_2,y_{\le 1},z_{\le 4})$. 
On the other hand, a quick inspection of the equations of $C_2^\o$ in $D(y_2)$
show that $x_4 \not\equiv 0$ on $C_2$, and since $C_2$ is contained 
$V(x_{\le 3}\setminus x_2,y_{\le 1},z_{\le 4})$, we conclude that
$C_2 \not\subset C_1$ in characteristic 2 as well. 
\end{step}

\appendix

\section{On the wedge lifting problem}
\label{ss:Appendix}

\begin{center}
\small
\textsc{Appendix by Tommaso de Fernex, Roi Docampo,
and Gebhard Martin}
\end{center}
\smallskip

Let $X$ be a surface defined over an uncountable algebraically closed field $k$.
A \emph{$k$-wedge} on $X$ is a morphism $\Phi \colon \Spec k[[s,t]] \to X$.
% where $K/k$ is a field extension. 
Denoting by $0$ and $\e$ the closed point and generic point of $\Spec k[[s]]$, 
one can view $\Phi$ as defining an arc $\Spec k[[s]] \to X_\infty$ on the arc space 
of $X$ giving a deformation of the \emph{special arc} $\f_0 := \Phi(0,t)$ 
of the wedge to its \emph{generic arc} $\f_\e := \Phi(\e,t)$.

Let $f \colon Y \to X$ denote the minimal resolution. 
As long as they are not fully contained in the indeterminacy locus
of $f^{-1}$, the arcs $\f_0$ and
$\f_\e$ lift individally to arcs $\~\f_0$ and $\~\f_\e$ on $Y$. 
The question is whether the wedge itself lifts to $Y$. 
More precisely, let $E$ be an irreducible component
of the exceptional locus of $f$, and denote by $V$ the set of $k$-rational 
points $P \in E$ with the property that every $k$-wedge $\Phi$ on $X$
such that $\~\f_0$ intersects $E$ transversally at $P$ 
automatically factors through the minimal resolution.
Then the \emph{$k$-wedge lifting problem} asks whether, for every $E$, 
the set $V$ has a nonempty intersection with $\bigcap_{n \in \N} U_n$
for every family $\{U_n\}_{n \in \N}$ of dense open subsets of $E$
\cite{LJ80,LJR12}.

The proof of the two-dimensional Nash problem given in \cite{dFD16}
solves affermatively the two-dimensional $k$-wedge lifting problem in 
characteristic zero. By contrast, Reguera showed that the $k$-wedge lifting problem fails 
for the $E_8$ singularity $x^2 + y^3 + z^5 = 0$ in
characteristics 2, 3, and 5 \cite{Reg12}. Here we extend
Reguera's characteristic 2 example to arbitrary characteristics.

\begin{proposition}
\label{p:WLP}
Let $p$ be a prime, and let 
$X \subset \A^3$ be the surface singularity defined by 
$x^p + y^{p+1} + z^{2p+1} = 0$ over an algebraically closed field 
$k$ of characteristic $p$.
For every $\lambda \in k \setminus \{0,-1\}$, the wedge $\Phi \colon \Spec k[[s,t]] \to X$ given by
\begin{align*}
\begin{cases}
x=-t^{2p+1}(s+ \lambda t)^{2p^2+p-1}(s+ (1 + \lambda)t)\\
y= t^{2p}(s+ \lambda t)^{2p^2-p}\\
z= t^p(s+ \lambda t)^{p^2}
\end{cases}
\end{align*}
does not lift to the minimal resolution of $X$. 
By varying $\lambda$, one obtains a $k$-rational wedge whose
special arc $\~\f_0$ intersects $E$ at any given $k$-rational
point away from its intersection with other exceptional 
components. 
\end{proposition}

For every pairwise coprime integers $a,b,c \ge 2$, 
a log resolution $f \colon Y \to X$ of the Brieskorn singularity 
$X = \{x^a + y^b + z^c = 0\} \subset \A^3$ over $\C$
was computed in \cite{HJ69,OW71}. 
The dual graph of the exceptional divisor of this resolution is a
star shaped tree with three branches. 
The central component is extracted by the normalized weighted blow-up 
with weights $(bc,ac,ab)$ in the coordinates $(x,y,z)$. 
This is the singular Seifert fibration of \cite{OW71}; 
in modern terminology it is the one-step embedded $\Q$-resolution 
by weighted blow-up \cite{MM13}. 
Resolving its three cyclic quotient singularities gives the 
Hirzebruch--Jung arms of the resolution graph.
While the classical references rely on topological methods, 
the resolution can be computed purely algebraically
and is defined over $\Z$ (see, e.g., \cite{LS23,AGMM24}).
The central component $E$ on the resolution $Y$ is a rational curve
with self intersection $-1$ or $-2$, depending on the exponents; 
the precise arithmetic conditions determining
which case occurs for each triple $(a,b,c)$ are explicitly given 
in \cite{HNK71}.

We now specialize to the case $(a,b,c) = (p, p+1, 2p+1)$, 
% \[
% a = p, \quad b = p+1, \quad c = 2p+1,
% \]
where $p$ is a prime.
In this case we have $E^2 = -2$, and $f$ is the minimal resolution of $X$. 
By adjoining the variable
\[
u := \frac{z^{2p+1}}{y^{p+1}}
\]
and suitably localizing, we obtain a chart on the weighted blow-up of $X$ 
given by $\Spec S$, where
\[
S := \frac{k[x,y,z,u,u^{-1},(1+u)^{-1}]}{(x^p + (1+u^{-1})z^{2p+1}, uy^{p+1} - z^{2p+1})}. 
\]
The normalization $\~S$ of $S$ gives a chart $\Spec \~S$ on $Y$
intersecting $E$. 

\begin{lemma}
\label{l:normalization}
The normalization $\~S$ of $S$ is a regular ring and is given by
$\~S = S[v]$ where
\[
v := \frac{xz^{2p-3}}{y^p}.
\]
For every $\lambda \ne 0,-1$, the pair $(u-\lambda ,v)$
gives a regular system of parameters in $\~S$
and $E$ is defined by $v=0$ in this chart.
\end{lemma}

\begin{proof}
Observing that 
\[
1+u = \frac{y^{p+1} + z^{2p+1}}{y^{p+1}} = - \frac{x^p}{y^{p+1}},
\]
one checks that 
\begin{align*}
v^{2p^2+3p+1} &= - u^{2p^2-p-3}(1+u)^{2p+3}x \\
v^{2p^2+p} &= - u^{2p^2-3p}(1+u)^{2p+1}y \\
v^{p^2+p} &= u^{p^2-p-1}(1+u)^{p+1}z.
\end{align*}
Any one of these equations shows that $v$ is integral over $S$. 
Solving for $(x,y,z)$, we obtain an isomorphism
\[
S[v] \simeq k[u,u^{-1},(1+u)^{-1},v].
\]
Since $k[u,u^{-1},(1+u)^{-1},v]$ is a regular ring, we conclude that $S[v]$ is
regular and is equal to $\~S$.
The last assertion is clear in view of the above isomorphism
and the fact that $\val_{E}(v) = 1$.
\end{proof}

\begin{proof}[Proof of \cref{p:WLP}]
Using that we are in characteristic $p$, 
it is immediate to check that the image of $\Phi$ is indeed
contained in $X$.

% We respectively denote by $0$ and $\e$ the closed point and generic point of $\Spec k[[s]]$.
As we shall see in a moment, the special arc $\f_0 := \Phi(0,t)$ of the wedge lifts
to an arc $\~\f_0$ on $Y$ which intersects transversally the divisor $E$ at a 
point $p$ that does not belong to any other exceptional divisor. 
The generic arc $\f_\e := \Phi(\e,t)$ of the wedge, on the other hand, lifts
to an arc $\~\f_\e$ on $Y$ which intersects some exceptional divisor $F$
other than $E$.

We work in the chart $\Spec \~S \subset Y$ 
and consider the elements
\[
u = \frac{z^{2p+1}}{y^{p+1}}, \quad v = \frac{xz^{2p-3}}{y^p}.
\]
The wedge lifts to the map
\[
\~\Phi \colon \Spec k[[st^{-1},t]] \to V
\]
given by
\begin{align*}
\begin{cases}
u = \dfrac{(s + \lambda t)^p}{t^p} \\
v = - \dfrac{(s + \lambda t)^{p-1}(s + (1+\lambda )t)}{t^{p-1}},
\end{cases}
\end{align*} 
where we regard these rational functions as polynomials in $s/t$ and $t$.  
From the equations, we see that  
$\Phi$ does not factor through $f$, hence it does not
lift to the minimal resolution. 
Geometrically, for $\lambda \in k \setminus\{0,-1\}$, 
the arc $\~\f_0$ intersects $E$
with order of contact 1 at the point $(u,v) = (\lambda^p,0)$, 
hitting any $k$-point of $E$ away from the other
exceptional components, 
and $\~\f_\e$ intersects a different component. 
\end{proof}

\end{document}